\documentclass[12pt]{amsart}
\usepackage{amssymb}
\usepackage{amsmath}
\usepackage{amsthm}
\usepackage{times}
\usepackage{graphicx}

\usepackage[curve,all]{xy}
 \xyoption{curve}
 \xyoption{line}
\xyoption{arc}

\usepackage{color}
\usepackage{amsthm}
\newtheorem{theorem}{Theorem}[section]
\newtheorem{lemma}[theorem]{Lemma}
\newtheorem{corollary}[theorem]{Corollary}

\newtheorem{prop}[theorem]{Proposition}

\theoremstyle{definition}

\theoremstyle{remark}
\newtheorem{remark}[theorem]{Remark}

\numberwithin{equation}{section}

\newcommand{\calA}{\mathcal{A}}
\newcommand{\calB}{\mathcal{B}}

\newcommand{\calO}{\mathcal{O}}

\def\Aut{{\text{Aut}}}

\def\PGL{{\text{PGL}}}

\def\Num{{\rm{Num}}}
\def\NS{{\rm{NS}}}

\def\deg{{\text{deg}}}

\def\Num{{\text{Num}}}

\begin{document}
\title [Enriques surfaces] {A 1-dimensional family of Enriques surfaces in characteristic 2 covered by the supersingular K3 surface with Artin invariant 1}

\author{Toshiyuki Katsura}
\address{Faculty of Science and Engineering, Hosei University,
Koganei-shi, Tokyo 184-8584, Japan}
\email{toshiyuki.katsura.tk@hosei.ac.jp}

\author{Shigeyuki Kond\=o}
\address{Graduate School of Mathematics, Nagoya University, Nagoya,
464-8602, Japan}
\email{kondo@math.nagoya-u.ac.jp}
\thanks{Research of the first author is partially supported by
Grant-in-Aid for Scientific Research (C) No. 24540053, and the second author by (S), No 22224001.}

\begin{abstract}
We give a 1-dimensional
family of classical and supersingular Enriques surfaces in characteristic 2 covered by the supersingular $K3$ surface with Artin invariant 1. Moreover we show that there exist  $30$ nonsingular
rational curves and ten non-effective $(-2)$-divisors on
these Enriques surfaces whose reflection group is of finite index in the orthogonal group of the N\'eron-Severi lattice modulo torsion.
\end{abstract}
\maketitle

\section{Introduction}\label{s1}
We work over an algebraically closed field $k$ of characteristic 2.
The main purpose of this paper is to give a 1-dimensional
family of Enriques surfaces in characteristic 2 covered by the supersingular $K3$ surface with Artin invariant 1.
In the paper \cite{BM2}, Bombieri and Mumford classified
Enriques surfaces into three classes, namely, singular, classical and supersingular Enriques surfaces.  
As in the case of characteristic $0$, an Enriques surface
$X$ in characteristic 2 has a canonical double cover
$\pi : Y \to X$.  The $\pi$ is a separable double cover,
a purely inseparable $\mu_2$- or $\alpha_2$-cover according to $X$ being singular, classical or supersingular. The surface $Y$ might have singularities, but it is $K3$-like in the sense
that its dualizing sheaf is trivial.  Bombieri and Mumford gave an explicit example of each type of Enriques surface
as a quotient of the intersection of three quadrics in ${\bf P}^5$. In particular, they gave an $\alpha_2$-covering $Y\to X$ such that $Y$ is a supersingular $K3$ surface with 12 rational double points of type $A_1$. 
Recently Liedtke \cite{L} showed that the moduli space of
Enriques surfaces with a polarization of degree $4$ has two $10$-dimensional irreducible components.  A general member of one component (resp. the other component) consists of singular (resp. classical) Enriques surfaces.  The intersection of two components parametrizes supersingular Enriques surfaces.
On the other hand, Ekedahl, Hyland and Shepherd-Barron
\cite{EHS} studied classical or supersingular Enriques surfaces whose canonical covers are supersingular $K3$ surfaces with 12 rational double points of type $A_1$.  They showed that
the moduli space of such Enriques surfaces is an open piece of
a ${\bf P}^1$-bundle over the moduli space of supersingular $K3$ surfaces.
Recall that the moduli space of supersingular $K3$ surfaces is 9-dimensional and
is stratified by Artin invariant $\sigma$, $1\leq \sigma \leq 10$. Each stratum has dimension $\sigma -1$ (Artin \cite{A}, Rudakov-Shafarevich \cite{RS2}).

In this paper, stimulating by Ekedahl, Hyland and Shepherd-Barron's work, we present a 1-dimensional family
of Enriques surfaces whose canonical covers are the (unique) supersingular $K3$ surface with Artin invariant 1.  These Enriques surfaces are parametrized by $a,b \in k,\ a+b=ab,\ a^3\not=1.$  If $a=b=0$, then the Enriques surface is supersingular, and otherwise it is classical (Theorem \ref{main}).  To construct these Enriques surfaces, we consider an elliptic surface defined
by
$$y^2+y+x^3 + s x(y^2 +y+1)=0$$
which has four singular fibers of type $I_3$ over
$s=1,\omega, \omega^2, \infty$ ($\omega^3=1, \omega\not=1$).
By taking Frobenius base change $s=t^2$, we have
an elliptic surface 
$$y^2+y+x^3 + t^2x(y^2 +y+1)=0.$$
which has 12 rational double points of type $A_1$ at
the singular points of each singular fiber.
By taking the minimal nonsingular model, we have an elliptic $K3$
surface $f : Y \to {\bf P}^1$ which is supersingular because
$f$ has four singular fibers of
type $I_6$ and hence its Picard number should be 22.
The Enriques surface $X=X_{a,b}$ is obtained as the quotient surface of $Y$ 
by a rational vector field
$$D = \frac{1}{t - 1}\left((t - 1)(t - a)(t - b)\frac{\partial}{\partial t} + (1 + t^2x)\frac{\partial}{\partial x}\right).
$$
The construction is based on a theory of inseparable double covering due to 
Rudakov-Shafarevich \cite{RS}
(see also Katsura-Takeda \cite{KT}).

The supersingular $K3$ surface $Y$ has Artin invariant $1$.  It was studied by
Dolgachev and the second author \cite{DK} (also see Katsura-Kondo \cite{KK}). It contains
42 nonsingular rational curves forming $(21)_5$-configuration.  These 42 curves are nothing but 24 components of four singular fibers of type $I_6$ and 18 sections of the fibration $f$.  
The automorphism group $\Aut(Y)$ is generated by
a subgroup $\PGL(3,{\bf F}_4)\cdot {\bf Z}/2{\bf Z}$ and
$168$ involutions associated with some $(-4)$-divisors on $Y$.  From this description, we see that
there exist 30 nonsingular rational curves and ten non-effective $(-2)$-divisors on the Enriques surface $X$ (see  Sections \ref{s5}, \ref{s6}).
The dual graph $\Gamma$ of these 40 divisors coincides with a graph 
obtained from an incidence relation between fifteen transpositions, fifteen permutations of type $(12)(34)(56)$ and
ten permutations of type $(123)(456)$ in the symmetric group $\mathfrak{S}_6$ of degree six.  Recall that fifteen transpositions are called Sylvester's duads and fifteen  permutations of type
$(12)(34)(56)$ Sylvester's synthemes (see Baker \cite{Ba}, p.220).   
It is possible
to choose a set of five synthemes which together contain all the
fifteen duads.  Such a family is called a total.  The number of 
possible totals is six.  And every two totals have one, and only one
syntheme in common between them.
We remark that there exist twelve (= six plus six) points on $X$ which have the following property:
if we denote by $1,2,\ldots, 6$ and $A, B, \ldots, F$ these twelve points suitably, then
the nodal curve corresponding to the transposition $ij$ passes the points $i$ and $j$, and 
the six points $A, B, \ldots, F$ can be considered as six totals so that 
the nodal curve corresponding to a syntheme appeared in two synthemes, for example, A and
B, passes the points $A$ and $B$ (see Section \ref{s5}).
Moreover these 40 divisors have the following remarkable property.  Let $\Num(X)=\NS(X)/{\rm \{ torsion\} }$ be the  N\'eron-Severi group of $X$ modulo torsion.  Then, together with the intersection pairing, it has a structure of an even unimodular lattice of signature $(1,9)$.  Let ${\rm O}(\Num(X))$ be the orthogonal group of the lattice $\Num(X)$ and
let $W(\Gamma)$ be the subgroup of ${\rm O}(\Num(X))$ generated by reflections associated with $40$ $(-2)$-divisors.
Then $W(\Gamma)$ is of finite index in ${\rm O}(\Num(X))$
(Theorem \ref{auto}).
This property will be helpful for determining the automorphism group $\Aut(X)$.

\medskip
\noindent
{\bf Acknowledgement.} The authors thank Shigeru Mukai for 
valuable conversations.

\section{Preliminaries}\label{sec2}

Let $k$ be an algebraically closed field of characteristic $p > 0$,
and let $S$ be a nonsingular complete algebraic surface defined over $k$.
We denote by $K_{S}$ a canonical divisor of $S$.
A rational vector field $D$ on $S$ is said to be $p$-closed if there exists
a rational function $f$ on $S$ such that $D^p = fD$. Let 
$\{U_{i} = {\rm Spec} A_{i}\}$ be an affine open covering of $S$. We set 
$A_{i}^{D} = \{D(\alpha) = 0 \mid \alpha \in A_{i}\}$. 
Affine varieties $\{U_{i}^{D} = {\rm Spec} A_{i}^{D}\}$ glue togather to 
define a normal quotient surface $S^{D}$.

Now, we  assume $D$ is $p$-closed. Then,
the natural morphism $\pi : S \longrightarrow S^D$ is a purely
inseparable morphism of degree $p$. 
If the affine open covering $\{U_{i}\}$ of $S$ is fine enough, then
taking local coordinate $x_{i}, y_{i}$
on $U_{i}$, we see that there exsit $g_{i}, h_{i}\in A_{i}$ and 
a rational function $f_{i}$
such that $g_{i} = 0$ and $h_{i} = 0$ have no common divisor,
and such that
$$
 D = f_{i}\left(g_{i}\frac{\partial}{\partial x_{i}} + h_{i}\frac{\partial}{\partial y_{i}}\right)
\quad \mbox{on}~U_{i}.
$$
By Rudakov-Shafarevich \cite{RS}, divisors $(f_{i})$ on $U_{i}$
give a global divisor $(D)$ on $S$, and zero-cycles defined
by the ideal $(g_{i}, h_{i})$ on $U_{i}$ give a global zero cycle 
$\langle D \rangle $ on $S$. A point contained in the support of
$\langle D \rangle $ is called an isolated singular point of $D$.
If $D$ has no isolated singular point, $D$ is said to be divisorial.
Rudakov and Shafarevich showed that $S^D$ is nonsingular
if and only if $\langle D \rangle  = 0$, i.e., $D$ is divisorial.
When $S^D$ is nonsingular,
they also showed a canonical divisor formula
\begin{equation}\label{canonical}
K_{S} \sim \pi^{*}K_{S^D} + (p - 1)(D),
\end{equation}
where $\sim$ means linear equivalence.
As for the Euler number $c_{2}(S)$ of $S$, we have a formula
\begin{equation}\label{euler}
c_{2}(S) = \deg \langle D \rangle  - \langle K_{S}, (D)\rangle - (D)^2
\end{equation}
(cf. Katsura-Takeda \cite{KT}). This is the dual version of Igusa's
formula (cf. Igusa \cite{I}).

Now we consider an irreducible curve $C$ on $S$ and we set $C' = \pi (C)$.
Take an affine open set $U_{i}$ above such that $C \cap U_{i}$ is non-empty.
The curve C is said to be integral with respect to the vector field $D$
if $(g_{i}\frac{\partial}{\partial x_{i}} + h_{i}\frac{\partial}{\partial y_{i}})$
is tangent to $C$ at a general point of $C \cap U_{i}$. Then, Rudakov-Shafarevich
\cite{RS} showed the following proposition:

\begin{prop}\label{insep}

$({\rm i})$  If $C$ is integral, then $C = \pi^{-1}(C')$ and $C^2 = pC'^2$.

$({\rm ii})$  If $C$ is not integral, then $pC = \pi^{-1}(C')$ and $pC^2 = C'^2$.
\end{prop}
In Section 4, we will use these results to construct Enriques surfaces
in characteristic 2.

A lattice is a free abelian group $L$ of finite rank equipped with 
a non-degenerate symmetric integral bilinear form $\langle . , . \rangle : L \times L \to {\bf Z}$. 
For a lattice $L$ and an integer 
$m$, we denote by $L(m)$ the free ${\bf Z}$-module $L$ with the bilinear form obtained from the bilinear form of $L$ by multiplication by $m$. The signature of a lattice is the signature of the real vector space $L\otimes {\bf R}$ equipped with the symmetric bilinear form extended from one on $L$ by linearity. A lattice is called even if 
$\langle x, x\rangle \in 2{\bf Z}$ 
for all $x\in L$. 
We denote by $U$ the even unimodular lattice of signature $(1,1)$, 
and by $A_m, \ D_n$ or $\ E_k$ the even {\it negative} definite lattice defined by
the Cartan matrix of type $A_m, \ D_n$ or $\ E_k$ respectively.    
We denote by $L\oplus M$ the orthogonal direct sum of lattices $L$ and $M$, and by $L^{\oplus m}$ the orthogonal direct sum of $m$-copies of $L$.
Let ${\rm O}(L)$ be the orthogonal group of $L$, that is, the group of isomorphisms of $L$ preserving the bilinear form.

\section{An elliptic pencil}\label{sec3}

From here on, throughtout this paper, we assume that $k$ is an algebraically closed field
of characteristic 2.
On the projective plane ${\bf P}^2$ over $k$, we consider the supersingular
elliptic curve $E$ defined by
$$
     x_1^2x_2 + x_1x_2^2 = x_0^3,
$$
where $(x_0, x_1, x_2)$ is a homogeneous coordinate of ${\bf P}^2$.
This is, up to isomorphism, the unique supersingular elliptic curve in characteristic 2.
The 3-torsion points of $E$ are given by
$$
\begin{array}{l}
Q_{0} =(0,1, 0), Q_{1} =(0, 0, 1), Q_{2} =(0,1, 1), Q_{3} =(1, \omega, 1),
 Q_{4} =(\omega,\omega, 1) \\
Q_{5} =(\omega^2,\omega, 1), Q_{6} =(1,\omega^2, 1), Q_{7} =(\omega,\omega^2, 1),
 Q_{8} =(\omega^2,\omega^2, 1).
\end{array}
$$
The point $Q_{0}$ is the zero point of $E$.
There exist 21 ${\bf F}_{4}$-rational points on ${\bf P}^2$, and
among them 9 points $Q_{i}$ ($i = 0, 1, \ldots, 8$) lie on $E$.
On the other hand, there exist 21 lines defined over ${\bf F}_{4}$
on ${\bf P}^2$, and among them 9 lines are triple tangents at $Q_{i}$
($i = 0, 1, \ldots, 8$) of $E$. Tangent lines intersect
$E$ only at the tangent points, and other lines intersect $E$
at 3 points among nine 3-torsion points transversely.

Now we consider the pencil of curves of degree 3, which pass through
nine points $Q_{i}$'s. Then the pencil is given by the equation
\begin{equation}\label{pencil}
x_1^2x_2 + x_1x_2^2 + x_0^3 + sx_0(x_1^2 + x_1x_2 + x_2^2) = 0
\end{equation}
with a parameter $s$. 
As is well-known, by blowing-ups at nine 3-torsion points 
we obtain an elliptic surface $\psi : R \longrightarrow {\bf P}^1$.
On the elliptic surface there exist 4 singular fibers of type $I_{3}$.
Five lines defined over ${\bf F}_{4}$ pass through
the point $Q_{i}$ on $E$. They consist of one triple tangent and four lines 
which intersect $E$ at $Q_{i}$ transversely. 
By the blowing-ups, the triple tangent line goes to the purely
inseparable double-section of the elliptic surface, and the 4 lines
go to components of four singular fibers respectively.
The 9 double sections pass through singular points of singular fibers
three-by-three.
The exeptional curves become nine sections of the elliptic surface
which pass through the regular points of singular fibers.
Each component of singular fibers intersects three sections among nine
exceptional curves.

\section{Construction of Enriques surfaces}\label{s4}

In characteristic 2, a minimal algebaic surface with numerically trivial
canonical divisor is called an Enriques surface if the second Betti
number is equal to 10. Such surfaces $S$ are  devided into three classes
(for details, see Bombieri-Mumford \cite{BM2}):
\begin{itemize}
\item[$({\rm i})$] $K_{S}$ is not linearly equivalent to zero 
and $2K_{S}\sim 0$.  Such an Enriques surface is called a classical Enriques surface.
\item[$({\rm ii})$] $K_{S} \sim 0$, ${\rm H}^{1}(S, {\calO}_{S}) \cong k$
and the Frobenius map acts on  ${\rm H}^{1}(S, {\calO}_S)$ bijectively.
Such an Enriques surface is called a singular Enriques surface.
\item[$({\rm iii})$] $K_{S} \sim 0$, ${\rm H}^{1}(S, {\calO}_{S}) \cong k$
and the Frobenius map is the zero map on  ${\rm H}^{1}(S, {\calO}_S)$.
Such an Enriques surface is called a supersingular Enriques surface.
\end{itemize}

Any elliptic fibration on a classical Enriques surface
has exactly two multiple fibers.  On the other hand, in
case of singular or supersingular Enriques surfaces,
any elliptic fibration has exactly one multiple fiber.

\begin{lemma}\label{lm;singular} 
Let $S$ be an Enriques surface. If there is a generically
surjective rational map from a supersingular $K3$ surface $\tilde{S}$ to $S$,
then $\tilde{S}$ is not a singular Enriques surface.
\end{lemma}
\begin{proof}
By Rudakov-Shafarevich \cite{RS}, $\tilde{S}$ is unirational.
Therefore, $S$ is also unirational. However,
a singular Enriques surface is not unirational by Crew \cite{Cr}
 (also see Katsura \cite{K}).
\end{proof}

In this section, we construct supersingular and classical Enriques surfaces, using 
the rational elliptic surface 
$\psi : R \longrightarrow {\bf P}^1$
constructed in Section 3 (see the equation (\ref{pencil})).
We consider the base change of $\psi : R \longrightarrow {\bf P}^1$
by $s = t^2$. Then we get an elliptic surface with 12 rational
double points of type $A_1$ defined by
\begin{equation}\label{pencil2}
x_1^2x_2 + x_1x_2^2 + x_0^3 + t^2x_0(x_1^2 + x_1x_2 + x_2^2) = 0.
\end{equation}
We consider the relatively minimal model of this elliptic surface (\ref{pencil2}):
\begin{equation}\label{pencil3}
f : Y \longrightarrow {\bf P}^1.
\end{equation}
From $Y$ to $R$, there exists a generically surjective 
purely inseparable rational map. Therefore, from $R^{(\frac{1}{2})}$
to $Y$, there also exists a generically surjective 
purely inseparable rational map. Since $R^{(\frac{1}{2})}$ is 
birationally isomorphic to ${\bf P}^2$,
we see that $Y$ is unirational. Hence, $Y$
is supersingular, i.e. the Picard number $\rho (Y)$ is equal 
to the second Betti number $b_{2}(Y)$ (cf. Shioda \cite{S}).

Now, we take an affine open set defined by $x_2 \neq 0$.
Then, on the affine open set this surface is defined by
$$
   y^2 + y + x^3 + t^2x(y^2 + y + 1) = 0.
$$
Considering the change of coordinates
$$
\begin{array}{l}
      v = (1 + t^3)\{(1 + t^2x)y + tx\}/t^6\\
      u = (1 + t^3)x/t^4
\end{array}
$$
we get a surface defined by
$$
v^2 + uv + t^2(t^4 + t)v + u^3 + (t^3 + 1)u^2 + t^2(t^4 + t)u = 0
$$
The discriminant of this elliptic surface is given by $\Delta(t) = t^6(t^3 + 1)^6$
(cf. Tate \cite{T}).
Therefore, we have $c_{2}(Y) = \sum_{t\in {\bf P}^1} {\rm ord}(\Delta(t)) = 24$, and
we conclude that $Y$ is a supersingular $K3$ surface (also see Dolgachev-Kondo \cite{DK}
and Katsura-Kondo \cite{KK}). We see there exist 4 singular fibers
of type $I_{6}$. These singular fibers exist over the points given by
$t = 1, \omega, \omega^2, \infty$.

For $f: Y \longrightarrow {\bf P}^1$,
there exist three exceptional curves derived from the resolution 
of the surface $(\ref{pencil2})$ on each singular fiber.
We denote them by $ E_{ij} (i = 1,\omega, \omega^2,  \infty; j = 1, 3, 5)$. 
We denote by $E_{ij} (i = 1,\omega, \omega^2,  \infty; j = 2, 4, 6)$ the rest of
components of singular fibers of $f : Y \longrightarrow {\bf P}^1$.
Here, $E_{i1}, E_{i2}, E_{i3}, E_{i4},  E_{i5},  E_{i6}$ are components 
of the singular fiber over $t = i$ $(i = 1,\omega, \omega^2,  \infty)$.
We have $E_{ij}^2 = -2$. Curves $E_{ij}$ and $E_{ij'}$ intersect each other transeversely 
if and only if $\mid j - j' ~({\rm mod ~6})\mid = 1$, and for other $j, j'$ we have
$\langle E_{ij}, E_{ij'}\rangle = 0$.

Now, we consider a rational vector field
$$
 D' = (t - 1)(t - a)(t - b)\frac{\partial}{\partial t} + (1 + t^2x)\frac{\partial}{\partial x}
\quad \mbox{with}~a + b = ab~\mbox{and}~a^3 \neq 1.
$$

\begin{lemma}\label{key}
Assume $a + b = ab$, $a^3 \neq 1$. Then,
\begin{itemize}
\item[$({\rm i})$] $D'^2 = t^2 D'$, namely, $D'$ is $2$-closed. 
\item[$({\rm ii})$] 
On the surface $Y$, the divisorial part of $D'$ is given by
$$
\begin{array}{rl}
(D') & = E_{11} + E_{13} + E_{15} -  E_{\omega1} - E_{\omega3} -  E_{\omega5}\\
 & \quad -  E_{\omega^21} - E_{\omega^23} -  E_{\omega^25} 
-  E_{\infty 2} - E_{\infty 4} -  E_{\infty 6} - F_{\infty},
\end{array}
$$
where $F_{\infty}$ is the fiber over the point given by $t = \infty$.
\item[$({\rm iii})$] The integral curves with respect to
$D$ are the following$:$ 
the smooth fibers over $t=a, b$ $($in case $a=b=0$, the smooth fiber over $t=0 )$ and
$$E_{12}, E_{14}, E_{16}, E_{\omega1}, E_{\omega3},  E_{\omega5}, E_{\omega^21}, E_{\omega^23}, E_{\omega^25}, 
 E_{\infty 2}, E_{\infty 4}, E_{\infty 6}.$$
\end{itemize}
\end{lemma}
\begin{proof}
These results follow from direct calculation. For example, to prove
(ii) and (iii), we consider a local chart of
the blowing-up at the point $(t, x, y)$ $= (1, 1, 0)$:
$$
  t + 1 = TU, ~x + 1 = U,~ y = VU
$$
with the new coordinates $T, U, V$.
Then, the exceptional curve $C$ is defined by $U = 0$ and
an irreducible component $C'$ of the fiber is given by $T = 0$ on the local chart.
We can show that the surface is nonsingular along $C$. 
It is easy to see that $T, U$ give
local coordinates on a neiborhood of $C$ in $Y$. Since 
$$
        \frac{\partial}{\partial t} =  \frac{1}{U}\frac{\partial}{\partial T},~ 
     \frac{\partial}{\partial x} =  \frac{\partial}{\partial U} + 
\frac{T}{U}\frac{\partial}{\partial T},
$$
on the local chart we have
$$
  D' = U\{(T^3 + (a + b)T^2)\frac{\partial}{\partial T} + (T^2U^2 + T^2U + 1)\frac{\partial}{\partial U}\}.
$$
Therefore, on the local chart we have the divisorial part $(D') = C$ and we see that
$C$ is not integral and $C'$ is integral with respect to the vector field $D'$.
On the other local charts for the blowing-ups, the calculation is similar.
\end{proof}

We set $D = \frac{1}{t - 1}D'$. Then, $D$ is also 2-closed, and we have
\begin{equation}\label{divisorial}
\begin{array}{rl}
   (D) &=  - (E_{12} + E_{14} + E_{16} +  E_{\omega1} + E_{\omega3} +  E_{\omega5}\\
 & \quad +  E_{\omega^21} + E_{\omega^23} +  E_{\omega^25} 
+  E_{\infty 2} + E_{\infty 4} +  E_{\infty 6}).
\end{array}
\end{equation}

\begin{lemma}
$Y^{D}$ is nonsingular.
\end{lemma}
 \begin{proof}
 We have
$$
\begin{array}{rl}
   (D)^2 &=  E_{12}^2 + E_{14}^2 + E_{16}^2 +  E_{\omega1}^2 + E_{\omega3}^2 +  E_{\omega5}^2\\
 & \quad +  E_{\omega^21}^2 + E_{\omega^23}^2 +  E_{\omega^25}^2 
+  E_{\infty 2}^2 + E_{\infty 4}^2 +  E_{\infty 6}^2\\
             & = (-2) \times 12 = -24
\end{array}
$$
Since $Y$ is a $K3$ surface, we have $c_{2}(Y) = 24$.
Therefore, by the equation (\ref{euler}), we have 
$$
24 = c_{2}(Y) = \deg \langle D\rangle - \langle K_{Y}, (D)\rangle - (D)^2 = \deg \langle D\rangle  + 24.
$$
Therefore, we have $\deg  \langle D\rangle = 0$ and $D$ is divisorial.
Hence, $Y^{D}$ is nonsingular.
By direct calculation we can also show that $D$ is divisorial.
\end{proof}

By the result on the canonical divisor formula of Rudakov and Shafarevich (see the equation (\ref{canonical})),
we have
$$
        K_{Y} = \pi^{*} K_{Y^D} + (D).
$$
\begin{lemma}
Let $C$ be an irreducible curve contained in the support of the divisor $(D)$,
and set $C' = \pi (C)$. Then, $C'$ is an exceptional curve of the first kind.
\end{lemma}
\begin{proof}
Since $C$ is integral with respect to $D$ (Lemma \ref{key}), 
we have $C = \pi^{-1}(C')$ (Proposition \ref{insep}).
Since $- 2 = C^2 = (\pi^{-1}(C'))^2 =2C'^2$, we have $C'^2 = -1$.
Since $Y$ is a $K3$ surface, $K_Y$ is 
linearly equivalent to zero.
Therefore, we have
$$
\begin{array}{rl}
2\langle K_{Y^D}, C'\rangle & =  \langle \pi^{*}K_{Y^D}, \pi^{*}(C')\rangle\\
    & =  \langle K_Y - (D), C\rangle = C^2 = -2.
\end{array}
$$
Hence we have $\langle K_{Y^D}, C'\rangle  = -1$.  Therefore,
the virtual genus of $C'$ is equal to $(\langle K_{Y^D}, C'\rangle + C'^2)/2 + 1 = 0$.
Hence, $C'$ is an exceptional curve of the first kind.
\end{proof}

We denote these 12 exceptional curves on $Y^{D}$ by $E'_{i}$ ($i = 1, 2, \ldots, 12$),
which are the images of irreducible components of $-(D)$ by $\pi$.
Now we have the following commutative diagram:
$$
\begin{array}{ccc}\label{maps}
       \quad    Y^{D} & \stackrel{\pi}{\longleftarrow} & Y \\
                \varphi \downarrow &    &   \downarrow f \\
      \quad        X     &        &   {\bf P}^1 \\
           g \downarrow & \quad   \swarrow_{F}&  \\
    \quad        {\bf P}^1 &   &
\end{array}
$$
Here, $\varphi$ is  the blowing-downs of $E'_{i}$ ($i = 1, 2, \ldots, 12$) and $F$ is the Frobenius base change.
Then, we have
$$
         K_{Y^D} = \varphi^{*}(K_{X}) + \sum_{i = 1}^{12}E'_{i}.
$$
\begin{lemma} 
The canonical divisor $K_{X}$ of $X$ is numerically equivalent to $0$.
\end{lemma}
\begin{proof}
By Lemma \ref{key}, all irreducible curves which appear
in the divisor $(D)$ are integral with respect to the vector field $D$.
For an irreducible component $C$ of $(D)$, we set $C' =\pi (C)$.
Then, we have $C = \pi^{-1}(C')$ (Proposition \ref{insep}). Therefore, we have
$$
       (D) = - \pi^{*}(\sum_{i = 1}^{12}E'_{i}).
$$
Since $Y$ is a $K3$ surface,
$$
\begin{array}{rl}
     0 \sim K_{Y} & = \pi^{*}K_{Y^D} + (D) \\
   &  = \pi^{*}( \varphi^{*}(K_{X}) + \sum_{i = 1}^{12}E'_{i})  + (D) = \pi^{*}(\varphi^{*}(K_{X}))
\end{array}
$$
Therefore, $K_{X}$ is numerically equivalent to zero.
\end{proof}

\begin{lemma} 
$b_{2}(X) = 10$ and $c_{2}(X) = 12$.
\end{lemma}
\begin{proof}
Since $\pi : {Y} \longrightarrow {Y}^{D}$ is finite and
purely inseparable, the \'etale cohomology of $Y$ is isomorphic to 
the \'etale cohomology of $Y^{D}$. Therefore, we have
$b_{1}(Y^{D}) = b_{1}(Y) = 0$, 
$b_{3}(Y^{D})= b_{3}(Y) = 0$ and $b_{2}(Y^{D}) 
= b_{2}(Y) = 22$. Since $\varphi$ is blowing-downs
of 12 exceptional curves of the first kind, we see
$b_{0}(X) =b_{4}(X) = 1$, $b_{1}(X) =b_{3}(X) = 0$ and $b_{2}(X) = 10$.
Therefore, we have 
$$
c_{2}(X) = b_{0}(X) - b_{1}(X) + b_{2}(X) -b_{3}(X) + b_{4}(X) = 12.
$$
\end{proof}

\begin{theorem}\label{main}
Under the notation above, the following statements hold.
\begin{itemize}
\item[$({\rm i})$] $X$ is a supersingular Enriques surface
if $a = b = 0$. 
\item[$({\rm ii})$] $X$ is a classical Enriques surface if $a + b = ab$ and 
$a \notin {\bf F}_{4}$.
\end{itemize}
\end{theorem}
\begin{proof}
(i) Assume $a = b = 0$. Then, the vector field $D$ is a fiber direction
only on the fiber over the point $P_{0}$ defined by $t = 0$ (Lemma \ref{key}).
Since $f^{-1}(P_{0})$ is a supersingular elliptic curve, the reduced part of
the fiber $g^{-1}(F(P_{0}))$ is also a supersingular elliptic curve,
and we have only one multiple fiber on the elliptic surface
$g : X \longrightarrow {\bf P}^{1}$.
Let $g^{-1}(F(P_{0})) = 2E_{0}$ be the multiple fiber. 
Then, since $E_{0}$ is a supersingular elliptic curve, 
it has no 2-torsion points. Therefore,
${\rm Pic}^{0}(E)$ has also no 2-torsion points. Since the normal bundle 
${\calO}(E_{0})\vert_{E_{0}} \in {\rm Pic}^{0}(E)$ and $({\calO}(E_{0})\vert_{E_{0}})^{\otimes 2}$
is a trivial invertible sheaf, ${\calO}(E_{0})\vert_{E_{0}}$ itself is trivial.
Therefore, $2E_{0}$ is a wild fiber (See Bombieri-Mumford \cite{BM1}, 
and Katsura-Ueno \cite{KU}).
The canonical divisor formula is given by
$$
\begin{array}{l}
     K_{X} = g^{*}(K_{{\bf P}^{1}} - L) + mE_{0}  \quad  \mbox{with an integer}
~m~ (0\leq m \leq 1), \\
      - \deg L = \chi(X, {\calO}_{X}) + t.
\end{array}
$$
Here, $t$ is the rank of the torsion part of ${\rm R}^{1}g_{*}{\calO}_{X}$.
There exist wild fibers if and only if $t \geq 1$ (cf. Bombieri-Mumford \cite{BM1}).
Since $2E_{0}$ is wild, we see $t \geq 1$. Since $K_{X}$ is numerically trivial
and $\deg K_{{\bf P}^{1}} = -2$,
considering the intersection of $K_{X}$ with a hyperplane section, we have
$$
    0 = (-2 + 1 + t) + \frac{m}{2}.
$$
Since $t \geq 1$ and $m \geq 0$, we conclude that $t = 1$ and $m = 0$.
Therefore, we have $ K_{X} \sim 0$.  
Since the second Betti number $b_{2}(X) = 10$, $X$ is either
singular Enriques surface or supersingular Enriques surface.
On the other hand, since $Y$ is a supersingular $K3$ surface, 
$X$ is not a singular Enriques surface by Lemma \ref{lm;singular}.
Hence, we conclude that $X$ is a supersingular Enriques surface.

(ii) We assume $a + b = ab$ and  $a \notin {\bf F}_{4}$.
Then, the vector field $D$ is a fiber direction
only on two fibers over the point $P_{a}$ defined by $t = a$ and over the point $P_{b}$
defined by $t = b$ (Lemma \ref{key}).
Let $g^{-1}(F(P_{b})) = 2E_{b}$ and  $g^{-1}(F(P_{a})) = 2E_{a}$
be two multiple fibers. 
Then, the canonical divisor formula is given by
$$
\begin{array}{l}
     K_{X} = g^{*}(K_{{\bf P}^{1}} - L) + m_{a}E_{a} + m_{b}E_{b} \\
    \quad \quad  \mbox{with integers}
~m_a~ \mbox{and}~m_{b}~ (0\leq m_a, m_{b} \leq 1) \\
      - \deg L = \chi(X, {\calO}_{X}) + t.
\end{array}
$$
Here, $t$ is the rank of the torsion part of ${\rm R}^{1}g_{*}{\calO}_{X}$.

Suppose both $E_{a}$ and $E_{b}$ are wild. Then we have $t \geq 2$.
Therefore, we have $\deg (K_{{\bf P}^{1}} - L) \geq -2 + 1  + 2 = 1$.
Hence, $K_{X}$ is not numerically equivalent to zero, a contradiction.

Now, suppose only one of $E_{a}$ and $E_{b}$, say $E_{b}$, is wild.
Then, $K_{X} = g^{*}(K_{{\bf P}^{1}} - L) + E_{a} + m_{b}E_{b} \quad 
 \mbox{with an integer}~m_{b}~ (0\leq m_{b} \leq 1)$ and $t \geq 1$. 
Then, we have $\deg (K_{{\bf P}^{1}} - L) \geq -2 + 1 + 1= 0$. 
Therefore, we have $K_{X} \succ E_{a}$ and $K_{X}$ is not numerically 
equivalent to zero, a contradiction.

Therefore, both $E_{a}$ and $E_{b}$ are tame, and the canonical divisor is given by
$$
   K_{X} = g^{*}(K_{{\bf P}^{1}} - L) + E_{a} + E_{b}\quad 
\mbox{with}~\chi(X, {\calO}_{X})  = 1, t = 0.
$$
Therefore, $K_{X}$ is not linearly equivalent to zero and $2K_{X} \sim 0$.
Since $b_{2}(X) = 10$, we conclude that $X$ is a classical Enriques surface.
\end{proof}

\section{$30$ nodal curves}\label{s5}

We use the same notation in the previous sections.  
We call a nonsingular rational curve on a $K3$ or an Enriques surface a nodal curve.  In this section and the next we will show that there exist 30 nodal curves and 10 non-effective $(-2)$-divisors on $X$.

First we recall some results for the supersingular $K3$ surface $Y$ with Artin invariant $1$ in Dolgachev-Kondo \cite{DK}.
The N\'eron-Severi lattice $\NS(Y)$ is an even lattice of
signature $(1,21)$ isomorphic to $U \oplus D_{20}$.  
The $K3$ surface $Y$ is obtained as the minimal resolution of 
a purely inseparable double cover 
$$p : Y \to {\bf P}^2$$
of the
projective plane ${\bf P}^2$.  
The purely inseparable double cover of ${\bf P}^2$ has
$21$ ordinary nodes over $21$ ${\bf F}_4$-rational points
${\bf P}^2({\bf F}_4)$.  Thus we have $21$ disjoint nodal curves on $Y$ as exceptional divisors.  On the other hand the pullbacks of $21$ lines in ${\bf P}^2({\bf F}_4)$
form 21 disjoint nodal curves on $Y$. 
Therefore
$Y$ contains 42 nodal curves.
These curves form a $(21)_5$-configuration, that is, they are divided into two families $\calA$ and $\calB$ each of which consists of $21$ disjoint curves, and 
each curve in one family meets exactly $5$ curves in another family at one point transverselly. 
Recall that $Y$ has a structure of an elliptic fibration 
$$f : Y \to {\bf P}^1$$
with four singular fibers of type $I_6$ and 
$18$ sections (see (\ref{pencil3})).  The above 42 nodal curves
coincide with the set of 24 irreducible components of singular fibers and 18 sections of the fibration $f$.

The action of the projective transformation group $\PGL(3,{\bf F}_4)$ on the plane can be lifted to automorphisms of $Y$.  Also there exists an involution $\sigma$ of $Y$, called a switch, changing two families $\calA$ and $\calB$.
The semi-direct product $\PGL(3, {\bf F}_4) \cdot {\bf Z}/2{\bf Z}$ preserves the 42 nodal curves.  Here ${\bf Z}/2{\bf Z}$ is generated by $\sigma$.  Moreover there exist 168 involutions of $Y$ as follows.  A set of six points in ${\bf P}^2({\bf F}_4)$ is called general if any three points in the set are not collinear.  There
are 168 general sets of six points.  For each general set of six points, we associate the Cremonat transformation
of the plane which can be lifted to an involution of
$Y$.  We call this involution the Cremona transformation associated with a general set $I$ of six points and denote it by $Cr_I$.  The action of $Cr_I$ on
$\NS(Y)$ is the reflection associated with a
$(-4)$-vector 
\begin{equation}\label{168}
2\ell - (C_1 + \cdots + C_6)
\end{equation}
in $\NS(Y)$.  Here $\ell$ is the class of the pullback of a line in the projective plane by $p$ and $C_1,\ldots, C_6$ are exceptional curves over the six points in $I$.
It is known that the group $\Aut(Y)$ is generated by $\PGL(3, {\bf F}_4)$, $\sigma$ and 168 Cremonat transformations (Dolgachev-Kondo \cite{DK}).
 
Let $X$ be the Enriques surface given in Theorem \ref{main}.  It is known that 
the N\'eron-Severi lattice
modulo torsions, denoted by $\Num(X)$, is isomorphic to
$U\oplus E_8$ which is an even unimodular lattice of 
signature $(1,9)$ (see Cossec-Dolgachev \cite{CD}).
Consider the map
$$\tilde{\pi} = \varphi \circ \pi : Y \to X$$ 
where $\pi : Y \to Y^D$ and $\varphi : Y^D \to X$ are given in Section \ref{s4}. 
Then $\tilde{\pi}^*(\Num(X))$ is a primitive sublattice
in $\NS(Y)$ isomorphic to $U(2)\oplus E_8(2)$
because $\langle \tilde{\pi}^*D, \tilde{\pi}^*D'\rangle$ 
$= 2\langle D, D'\rangle$.
Denote by $E_1,\ldots, E_{12}$ the 12 disjoint integral nodal curves on $Y$ which are contracted under the map $\tilde{\pi}$ (In the equation (\ref{divisorial}) in Section \ref{s4}, we denote them by
$E_{12}$, $E_{14}$, $E_{16}$, $E_{\omega1}$, $E_{\omega3}$,  
$E_{\omega5}$, $E_{\omega^21}$, $E_{\omega^23}$, $E_{\omega^25}$, 
$E_{\infty 2}$, $E_{\infty 4}$, $E_{\infty 6}$).  
Note that these 12 curves consist of
6 curves in $\calA$ and 6 curves in $\calB$.
Let $A_1^{\oplus 12}$ be the sublattice in $\NS(Y)$ generated by
$E_1, \ldots, E_{12}$.  Obviously $A_1^{\oplus 12}$ is orthogonal to $\tilde{\pi}^*(\Num(X))$.

As mentioned above, there are 42 nodal curves on
$Y$.  Among them, 12 curves $E_1, \ldots, E_{12}$ are integral and contracted by $\tilde{\pi}$.  In the following we discuss the remaining 30 non-integral curves.  
Let $F$ be a remaining non integral nodal curve. Note that $F$ meets exactly 
two curves among $E_1, \ldots, E_{12}$ and 
the image $\pi(F)$ has the self-intersection number $-4$
by Proposition \ref{insep}.
The image $\tilde{\pi}(F)$ is a nodal curve.
Let $F'$ be an another remaining curve.  If $\langle F, F'\rangle =1$, then $\tilde{\pi}(F)$ meets $\tilde{\pi}(F')$ at one point with multiplicity 2.
Assume that $F$ belongs to
the family $\calA$. 
Recall that $F$ meets 5 curves in $\calB$.
Denote by $E, E', F_1, F_2, F_3$ the curves meeting with $F$ where $E, E'$ are integral, that is, they belong to $\{E_1, \ldots , E_{12}\}$.  Assume that $E$ meets $F, G_1,\ldots, G_4$ and $E'$ meets 
$F, G_1', \ldots, G_4'$. Obviously $G_1,\ldots, G_4, G_1',\ldots, G_4'$ belong to $\calA$. Then 
the image $\tilde{\pi}(F)$ meets three curves $\tilde{\pi}(F_i)$ $(i=1,2,3)$ with multiplicity 2 and meets 4 curves 
$\tilde{\pi}(G_i), 1\leq i\leq 4,$ (resp. $\tilde{\pi}(G_i'), 1\leq i\leq 4$)
at the point $\tilde{\pi}(E)$ (resp. $\tilde{\pi}(E')$).
We now get the following lemma.

\begin{lemma}
There exist $30$ nodal curves on $X$ which are the images
of the $30$ nodal curves not belonging to
$\{E_1,\ldots, E_{12}\}$.  
Let $\bar{\calA}$ and $\bar{\calB}$ be the families of nodal curves which are the images of curves in $\calA$ and $\calB$ respectively.  Each nodal curve in one family 
tangents three nodal curves in another family.
Each nodal curve $C$ 
in one family meets $8$ nodal curves in
the same family transversally.  Moreover $4$ curves in
these $8$ nodal curves meet at a point on $C$ and
the remaining $4$ curves meet at another point on $C$.
\end{lemma}

In the following we show that the incidence relation between nodal curves in 
$\bar{\calA}$ and $\bar{\calB}$ is the same as that of
Sylvester's duads and synthemes.
First we recall Sylvester's duads and synthemes (see Baker \cite{Ba}, p.220). 
We denote by $ij$ the transposition of $i$ and $j$ 
($1 \leq i \not= j \leq 6$) which is classically called Sylvester's duad.
Six letters $1,2,3,4,5,6$ can be arranged in three pairs of duads,
for example, (12, 34, 56), called Sylvester's syntheme.  
(It is understood
that (12, 34, 56) is the same as (12, 56, 34) or (34, 12, 56)).
Duads and Synthemes are in $(3,3)$ correspondence, that is, 
each syntheme consists
of three duads and each duad belongs to three synthemes.  
 It is possible
to choose a set of five synthemes which together contain all the
fifteen duads.  Such a family is called a total.  The number of 
possible totals is six.  And every two totals have one, and only one
syntheme in common between them.  The following table gives 
the six totals $A, B, \ldots, F$ in its rows, and also 
in its columns (see Baker [Ba], p.221) :

\medskip

\halign{\hfil\tt#\hfil&&\quad#\hfil\cr
& & A & B & C & D & E & F \cr
\noalign{\smallskip}
& A &  & 14,25,36 & 16,24,35 & 13,26,45 & 12,34,56& 15,23,46 \cr
\noalign{\smallskip}
& B & 14,25,36& & 15,26,34& 12,35,46& 16,23,45& 13,24,56\cr
\noalign{\smallskip}
& C & 16,24,35& 15,26,34&  & 14,23,56& 13,25,46& 12,36,45\cr
\noalign{\smallskip}
& D & 13,26,45& 12,35,46& 14,23,56& & 15,24,36& 16,25,34\cr
\noalign{\smallskip}
& E & 12,34,56& 16,23,45& 13,25,46& 15,24,36& & 14,26,35\cr
\noalign{\smallskip}
& F & 15,23,46& 13,24,56& 12,36,45& 16,25,34& 14,26,35& \cr}

\medskip
\noindent
Now we consider six letters $1,\ldots, 6$ as the six points on $X$ which are the images of curves in $\calA$ contracted by $\tilde{\pi}$, and six totals $A,\ldots, F$
as the six points on $X$ which are the images of curves in $\calB$ contracted by $\tilde{\pi}$.
Also consider fifteen duads as fifteen nodal curves in $\bar{\calA}$.  The transposition $ij$ corresponds to the nodal curve through the two points $i$ and $j$.
On the other hand, consider fifteen synthemes as fifteen nodal curves in $\bar{\calB}$.  A syntheme 
corresponds to the nodal curve through the two points
corresponding to two totals containg the syntheme.
Then two curves in $\bar{\calA}$ meet if the corresponding
two duads have a common letter, and two curves in 
$\bar{\calB}$ meet if the corresponding two synthemes have
no common duads.  And
the $(3,3)$ correspondence between duads and synthemes
describes the intersection relation between fifteen curves
in $\bar{\calA}$ and fifteen curves in $\bar{\calB}$.  For example, the nodal curve
$(12,34,56)$ tangents to nodal curves $12, 34, 56$ and meets eight nodal curves
in $\bar{\calB}$ belonging to the totals $A$ or $E$ at the points $A$ and $E$ (see Figure \ref{enriques12}).
The nodal curve $12$ tangents to nodal curves $(12, 34, 56), (12,35,46), (12,36,45)$ and meets eight nodal curves in $\bar{\calA}$ containing the letter $1$ or $2$ at the points $1$ and $2$.
Thus fifteen duads, fifteen synthemes, six letters and six totals are realized on
the Enriques surface $X$ geometrically.

\begin{figure}[!htb]
 \begin{center}
  \includegraphics[width=120mm]{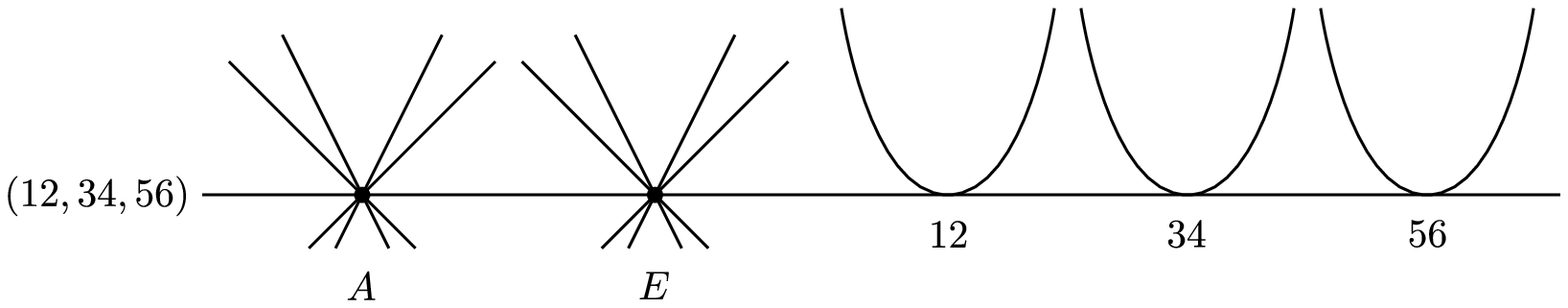}
 \end{center}
 \caption{}
 \label{enriques12}
\end{figure}

\section{Ten $(-2)$-divisors}\label{s6}

We keep the same notation in the previous section.  Recall that the $K3$ surface
$Y$ has 168 divisors given in $(\ref{168})$.

\begin{lemma}\label{10reflections}
There exist ten divisors among $168$ divisors which are orthogonal to $A_1^{\oplus 12}$ generated by $E_1,\ldots, E_{12}$.
\end{lemma}
\begin{proof}
For simplicity, we assume that $E_1,\ldots, E_6$ are the pullbacks of six lines 
$\ell_1, \ldots, \ell_6$ in ${\bf P}^2({\bf F}_4)$ and
$E_7,\ldots, E_{12}$ are exceptional curves over ${\bf F}_4$-rational points $p_1,\ldots, p_6$ of ${\bf P}^2$.  Obviously $p_1,\ldots, p_6$ do not lie on $\ell_i$, $1\leq i \leq 6$.  Moreover the set $\{ p_1,\ldots, p_6\}$ of six points is general by construction.
Let $\tilde{r} = 2\ell - (C_1 + \cdots + C_6)$ be a divisor such that $C_1,\ldots,C_6$ are exceptional curves over
general six points $q_1,\ldots, q_6$ on ${\bf P}^2({\bf F}_4)$.  Assume that $\langle \tilde{r}, E_j\rangle =0$ for
$j = 1,\ldots, 12$.  Since $\langle \ell, C_i\rangle =0$, we see $\langle \tilde{r}, C_i\rangle =2$.  Hence we have $E_j\not= C_i$ $(i=1,\ldots, 6; j=7,\ldots, 12)$.
The condition $\langle \tilde{r}, E_j\rangle =0$ implies that each $E_j$ $(j=1,\ldots, 6)$ meets exactly two curves
in $\{ C_1,\ldots, C_6\}$.   This means that the six points $q_1,\ldots, q_6$ are intersection points of six lines
$\ell_1,\ldots, \ell_6$.  Thus the divisors $\tilde{r}$ satisfying $\langle \tilde{r}, E_j\rangle =0$ $(j=1,\ldots, 12)$
correspond to the set of general six points $q_1,\ldots, q_6$ which are intersections between $\ell_1,\ldots, \ell_6$. 
We will show that 
six lines $\ell_1,\ldots, \ell_6$ are divided into two sets $\{ \ell_i, \ell_j, \ell_k\}$ and $\{\ell_l, \ell_m, \ell_n\}$ such that 
six points $q_1,\ldots, q_6$ coincide with the intersection points of three lines $\ell_i, \ell_j, \ell_k$ and those of $\ell_l, \ell_m, \ell_n$. 
Denote by $ij$ the intersection point of $\ell_i$ and $\ell_j$.
If six points are given by $ij, jk, ki, mn, nl, lm$, then
we have the desired one.  Otherwise six points are given by $ij, jk, kl, lm, mn, ni$ because each letter appears twice.  In this case, the line $\ell$ through $ij$ and $kl$ does not appear in 
$\{\ell_1,\ldots, \ell_6\}$.  Since the set $\{ p_1,\ldots, p_6\}$ of six points is general, $\ell$ passes exactly two points in $\{ p_1,\ldots, p_6\}$.
Since $\ell$ contains five ${\bf F}_4$-rational points,
it should pass one more point not lying on $\ell_i\cup \ell_j \cup \ell_k\cup \ell_l$ because $\ell  \cap \{ \ell_i\cup \ell_j \cup \ell_k\cup \ell_l\} = \{ij, kl\}$.
This implies that $\ell$ passes the remaining point $mn$.  This contradicts the generality of
the six points $ij, jk, kl, lm, mn, ni$.
Thus we have the assertion.
\end{proof}

Let $\tilde{r}_a, \tilde{r}_b, \ldots, \tilde{r}_j$ be ten
divisors in $\NS(Y)$ indexed by ten letters $a,b, \ldots, j$ which are given in Lemma \ref{10reflections}.
Let $r_a, r_b,\ldots,
r_j \in \Num(X)$ be the images of $\tilde{r}_a, \tilde{r}_b, \ldots, \tilde{r}_j$. 
Since $\tilde{r}_a^2 =\cdots = \tilde{r}_j^2 = -4$, 
we have $r_a^2 =\cdots = r_j^2 =-2$.
Consider two distinct divisors $\tilde{r}$ and $\tilde{r}'$.
Assume that $\tilde{r}$ (resp. $\tilde{r}'$) correspond to six points $q_1,..., q_6$ (resp. $q_1',..., q_6'$)
which are the union of intersection points of $\ell_i, \ell_j, \ell_k$ and those of $\ell_l, \ell_m, \ell_n$
(resp. the union of intersections of $\ell_{i'}, \ell_{j'}, \ell_{k'}$ and those of $\ell_{l'}, \ell_{m'}, \ell_{n'}$).
Note that either $|\{i, j, k\}\cap \{i', j', k'\}| =2$ or $|\{i, j, k\}\cap \{l', m', n'\}|=2$.  This implies that
$$|\{q_1,\ldots, q_6\} \cap \{q_1',\ldots, q_6'\}| = 2.$$
Therefore we have $\langle \tilde{r}_a, \tilde{r}_b\rangle = 4$, and hence $\langle r_a, r_b\rangle = 2$.  Thus 
we have the following Lemma.

\begin{lemma}\label{complete}
The dual graph of $\{r_a, r_b,\ldots, r_j\}$ is a complete graph whose
edges are double lines.
\end{lemma}

Now, we discuss the incidence relation between ten $(-2)$-vectors $r_a,\ldots, r_j$ and 
fifteen duads, fifteen synthemes.

\begin{lemma}\label{10divisors}
Each vector in $\{r_a,\ldots, r_j\}$ meets exactly six duads and six synthemes with intersection multiplicity two.
\end{lemma}
\begin{proof}
We use the same notation as in the proof of Lemma \ref{10reflections}.  
Let $C$ be the nodal curve on $Y$ corresponding to a duad.
Then $C$ meets exactly two nodal curves $E, E'$ in $\{E_1,\ldots. E_6\}$.  Then $2C + E+E'$ is perpendicular to
$A_1^{\oplus 12}$, that is, $2C + E+E' \in \tilde{\pi}^*(\Num(X)) = U(2)\oplus E_8(2)$.  Let 
$\tilde{r} = 2\ell - (C_1 + \cdots + C_6)$ be  a divisor in $\{\tilde{r}_a,\ldots, \tilde{r}_j\}$.
Then $\langle E, C_1 + \cdots + C_6\rangle = \langle E', C_1 + \cdots + C_6\rangle = 2$.
If $C$ appears in $\{ C_1,\ldots, C_6\}$, then 
$$\langle \tilde{r}, 2C + E+E'\rangle = 4,$$
and if $C$ does not appear in $\{ C_1,\ldots, C_6\}$, then 
$$\langle \tilde{r}, 2C + E+E'\rangle = 0.$$
The proof for which $C$ corresponds to a syntheme is similar.
Thus we have the assertion.
\end{proof}

\noindent
We can identfy ten divisors $r_a,\ldots, r_j$ with
ten symbols 
$$(123,456), (124,356),(125,346),(126,345), (134,256),$$
$$(135,246),(136,245),(145,236),(146,235),(156,234).$$
For example, $(123,456)$ meets six duads $12, 13,23,45,46,56$ and six synthemes 
$$(14,25,36), (14,26,35),
(15,24,36),(15,26,34),(16,24,35),(16,25,34).$$

\medskip
We denote by $\Gamma$ the dual graph of 30 nodal curves and
ten $(-2)$-divisors.

\begin{remark}
The graph $\Gamma$ appears in other places.  For example, consider the moduli space of
principally polarized abelian surfaces with level $2$-structure over the field ${\bf C}$
of complex numbers.  It has fifteen $0$-dimensional and fifteen $1$-dimensional boundary components and contains ten divisors
parametrizing abelian surfaces of product type {\rm (e.g. see \cite{vG})}.  On the other hand,
S. Mukai found the existence of the above configuration of $30$ nodal curves and ten $(-2)$-vectors on an Enriques surface defined over
${\bf C}$ {\rm (}unpublished{\rm )}.  
\end{remark}

\begin{prop}
The automorphism group of the graph $\Gamma$ is
isomorphic to the automorphism group $\Aut(\mathfrak{S}_6)$ of the symmetric group $\mathfrak{S}_6$ of degree $6$.
\end{prop}
\begin{proof}
Recall that $\Aut(\mathfrak{S}_6)$ is generated by $\mathfrak{S}_6$ and an outer automorphism.  An outer automophism interchanges
duads with synthemes, and six letters $1,\ldots ,6$ with six totals $A,\ldots , F$ respectively.
Obviously $\Aut(\mathfrak{S}_6)$ preserves the graph $\Gamma$.  Let $g$ be an automorphism of $\Gamma$.  If necessary, by compositing an outer automorphism, we assume $g$ preserves six letters.  If $g$ fixes each of six letters, then $g$ acts on $\Gamma$ identically.  Thus $g$ is contained in $\mathfrak{S}_6$.
\end{proof}

\begin{remark}
The N\'eron-Severi lattice $\NS(Y)$ is isomorphic to the orthogonal complement of the root lattice $D_4$ in the even unimodular lattice $II_{1,25}$ of signature $(1,25)$.
If we embed $\NS(Y)$ into $II_{1,25}$ as the orthogonal complement, then $42$ nodal curves and $168$ $(-4)$-divisors on $Y$ are the projections of Leech roots into $\NS(Y)$ {\rm (see \cite{DK}, \S 3.3)}.  The lattice $\tilde{\pi}^*(\Num(X))$ $(\cong U(2)\oplus E_8(2))$ is the orthogonal complement of $D_4\oplus A_1^{\oplus 12}$ in $II_{1,25}$, and the above $30$ nodal curves and $10$ $(-2)$-divisors on $X$ correspond to the projections of some Leech roots.
\end{remark}

\section{Automorphisms}\label{s7}
Let $S$ be an Enriques surface.  
Let $\Num(S)$ be the N\'eron-Severi lattice
modulo torsions. Then $\Num(S)$ is an even unimodular lattice of signature $(1,9)$ (Cossec-Dolgachev \cite{CD}).
We denote by ${\rm O}(\Num(S))$ the orthogonal group of $\Num(S)$. The set 
$$\{ x \in \Num(S)\otimes {\bf R} \ : \ \langle x, x \rangle > 0\}$$ 
has two connected components.
Denote by $P(S)$ the connected component containing an ample class of $S$.  
For $\delta \in \Num(S)$ with $\delta^2=-2$, we define
an isometry $s_{\delta}$ of $\Num(S)$ by
$$s_{\delta}(x) = x + \langle x, \delta\rangle \delta, \quad x \in \Num(S).$$ 
The $s_{\delta}$ is called the reflection associated with $\delta$.
Let $W(S)$ be the subgroup of
${\rm O}(\Num(S))$ generated by reflections associated with all nodal curves on $S$.  Then $P(S)$ is divided into chambers 
each of which is a fundamental domain with respect to
the action of $W(S)$ on $P(S)$.
There exists a unique chamber containing an ample
class which is nothing but the closure of the ample cone $D(S)$ of $S$.
It is known that $\Aut(D(S))$ is isomorphic to 
the quotient group ${\rm O}(\Num(S))/\{\pm 1\}\cdot W(S)$.
The natural map
$$\Aut(S) \to \Aut(D(S))$$
is isomorphic up to finite groups, that is, it has finite kernel and cokernel (e.g. Dolgachev \cite{D}).
In particular $\Aut(S)$ is finite if and only if
$W(S)$ is of finite index in ${\rm O}(\Num(S))$.
Over the field of complex numbers, Enriques surfaces with finite
group of automorphisms were classified by Nikulin \cite{N}
and the second author \cite{Ko}.  
In general it is difficult to describe the group
$\Aut(D(S))$.  

Now, we recall Vinberg's criterion for
which a group generated by finite number of reflections is
of finite index in ${\rm O}(\Num(S))$.

Let $\Delta$ be a finite set of $(-2)$-vectors in $\Num(S)$.
Let $\Gamma$ be the graph of $\Delta$, that is,
$\Delta$ is the set of vertices of $\Gamma$ and two vertices $\delta$ and $\delta'$ are joined by $m$-tuple lines if $\langle \delta, \delta'\rangle=m$.
We assume that the cone
$$K(\Gamma) = \{ x \in \Num(S)\otimes {\bf R} \ : \ \langle x, \delta_i \rangle \geq 0, \ \delta_i \in \Delta\}$$
is a strictly convex cone. Such $\Gamma$ is called non-degenerate.
A connected parabolic subdiagram $\Gamma'$ in $\Gamma$ is a  Dynkin diagram of type $\tilde{A}_m$, $\tilde{D}_n$ or $\tilde{E}_k$ (see \cite{V}, p. 345, Table 2).  If the number of vertices of $\Gamma'$ is $r+1$, then $r$ is called the rank of $\Gamma'$.  A disjoint union of connected parabolic subdiagrams is called a parabolic subdiagram of $\Gamma$.  The rank of a parabolic subdiagram is the sum of the rank of its connected components.  Note that the dual graph of singular fibers of an elliptic fibration on $Y$ gives a parabolic subdiagram.  For example, a singular fiber of type $III$, $IV$ or $I_{n+1}$ defines a parabolic subdiagram of type $\tilde{A}_1$, $\tilde{A}_2$ or 
$\tilde{A}_n$ respectively.  
We denote by $W(\Gamma)$ the subgroup of ${\rm O}(\Num(S))$ 
generated by reflections associated with $\delta \in \Gamma$.

\begin{prop}\label{Vinberg}{\rm (Vinberg \cite{V}, Theorem 2.3)}
Let $\Delta$ be a set of $(-2)$-vectors in $\Num(S)$
and let $\Gamma$ be the graph of $\Delta$.
Assume that $\Delta$ is a finite set, $\Gamma$ is non-degenerate and $\Gamma$ contains no $m$-tuple lines with $m \geq 3$.  Then $W(\Gamma)$ is of finite index in ${\rm O}(\Num(S))$ if and only if every connected parabolic subdiagram of $\Gamma$ is a connected component of some
parabolic subdiagram in $\Gamma$ of rank $8$ {\rm (}= the maximal one{\rm )}.
\end{prop} 
\noindent
For the proof of Proposition \ref{Vinberg}, see Vinberg \cite{V} (also see \cite{Ko}, Theorem 1.9).

\medskip
Let $X$ be the Enriques surface given in Theorem \ref{main}.
In the following,
as $\Delta$ we take $40$ $(-2)$-vectors in $\Num(X)$ corresponding to
fifteen duads, fifteen synthemes and ten $(-2)$-vectors given in
the previous section.  Let $\Gamma$ be the graph of
these 40 vectors.  We directly see the following Lemma.

\begin{lemma}\label{parabolic}
The maximal parabolic subdiagrams of $\Gamma$ are
$$\tilde{A}_2\oplus \tilde{A}_2\oplus \tilde{A}_2\oplus \tilde{A}_2,\  \tilde{A}_4\oplus \tilde{A}_4,\ \tilde{A}_5\oplus \tilde{A}_2\oplus \tilde{A}_1,\ \tilde{A}_3\oplus \tilde{A}_3 \oplus \tilde{A}_1 \oplus \tilde{A}_1$$
each of which has the maximal rank $8$. 
\end{lemma}

In the following we give an example of each maximal parabolic subdiagrams.

\medskip
\noindent
(i) \ The diagram $\tilde{A}_2\oplus \tilde{A}_2\oplus \tilde{A}_2\oplus \tilde{A}_2$ corresponds to an elliptic fibration on $X$ with four singular fibers of type $I_3$.
For example, four sets 
$$\{ 12, 23, 13\}, \{ 45, 46, 56\}, \{ (14,25,36), (15,26,34),(16,24,35)\},$$ 
$$\{ (14,26,35), (15,24,36), (16,25,34)\}$$ 
are components of
singular fibers of an elliptic fibration of this type.
The syntheme $(12,35,46)$ is a 2-section of this fibration.

\medskip
\noindent
(ii) \ The diagram $\tilde{A}_4\oplus \tilde{A}_4$ corresponds to an elliptic fibration on $X$ with two singular fibers of type $I_5$.
For example, two sets $\{ 12,23,34,45,15\}$ and
 
\noindent
$\{ (13,25,46), (14,26,35),(13,24,56),(14,25,36),(16,24,35)\}$ are components of singular fibers of an elliptic fibration  and the duad $46$ is a 2-section of this fibration.

\medskip
\noindent
(iii) \ The diagram $\tilde{A}_5\oplus \tilde{A}_2\oplus \tilde{A}_1$ corresponds to an elliptic fibration on $X$ with singular fibers of type $I_6$, $IV$ and $I_2$.  For example, six synthemes 
$$(14,25,36), (15,26,34), (14,23,56), (15,24,36), (14,26,35),(15,23,46)$$ 
are components of a singular fiber of type $I_6$, three duads $12,13,16$ are components of a singular fiber of type $IV$.  The pair of the duad $45$ and
$(-2)$-vector $(145,236)$ forms the subdiagram of type
$\tilde{A}_1$. The duad $56$ is a 2-section of this fibration.

\begin{remark}
Note that there exists a
nodal curve $C$ such that $C$ and the duad $45$ form the
singular fiber of type $I_2$. 
If we denote by $2f$ the class of a multiple fiber of this fibration, then
$$(145,236) = f - C.$$
The $2$-section $56$ meets $C$, but not $(145,236)$. 
Note that $C$ does not appear in $40$ $(-2)$-vectors. 
\end{remark}

\medskip
\noindent
(iv) \ The diagram $\tilde{A}_3\oplus \tilde{A}_3 \oplus \tilde{A}_1 \oplus \tilde{A}_1$ corresponds to an elliptic fibration on $X$ with two singular fibers
of type $I_4$ and one singular fiber of type $III$.
For example, four duads $24,25,34,35$ and four synthemes
$$(12,36,45), (14,23,56), (13,26,45),(15,23,46)$$ 
define
two singular fibers of type $I_4$ respectively, and the pair of the duad $16$ and the syntheme $(16,23,45)$ defines a singular fiber of type $III$.  The remaining subdiagram of type $\tilde{A}_1$ consists of
two $(-2)$-vectors $(123,456)$ and $(145,236)$. The duad $13$ is a 2-section of this fibration.  

Denote by $D(\Gamma)$ the finite polyhedron defined by
40 $(-2)$-vectors in $\Gamma$.  Combining Proposition \ref{Vinberg} and Lemma \ref{parabolic}, we have the following theorem. 

\begin{theorem}\label{auto}
The group $W(\Gamma)$ is of  finite index in ${\rm O}(\Num(X))$, and
$$\Aut(D(\Gamma)) (\cong {\rm O}(\Num(X))/\{\pm 1\}\cdot W(\Gamma))$$ 
is isomorphic to the semi-direct product $\mathfrak{S}_6\cdot {\bf Z}/2{\bf Z}$ where $\mathfrak{S}_6$ is the symmetric group of the six letters $\{1, \ldots, 6\}$ and ${\bf Z}/2{\bf Z}$ is generated by an outer automorphism of $\mathfrak{S}_6$.   
\end{theorem}

Recall that $\Aut(Y)$ is generated by $\PGL(3,{\bf F}_4)$, a switch and 
168 Cremonat transformations, where $Y$ is the covering $K3$ surface of $X$.
Among these automorphisms, the subgroup $\mathfrak{S}_6\cdot  {\bf Z}/2{\bf Z}$ and ten Cremonat transformations preserve 12 nodal curves
$E_1,\ldots, E_{12}$.

\medskip
\noindent
{\bf Conjecture.}   {\it The subgroup $\mathfrak{S}_6\cdot  {\bf Z}/2{\bf Z}$ and 
ten Cremonat transformations descend to automorphisms of $X$. }

\medskip
\noindent
Let $G$ be the subgroup of ${\rm O}(\Num(X))$ generated by reflections associated with ten non-effective divisors in $\Gamma$.  If the conjecture is true, then ten Cremonat
transformations descend to ten generators of $G$.
By an argument in Vinberg \cite{V2}, 1.6, we have the following Corollary. 
\begin{corollary} 
Assume the conjecture holds.  Then $\Aut(X)$ is generated 
by $\Aut(D(\Gamma)) (\cong \mathfrak{S}_6\cdot  {\bf Z}/2{\bf Z})$ and $G$, up to finite groups.
\end{corollary}


\begin{thebibliography}{99}
\bibitem{A} M.\ Artin, \newblock{Supersingular {$K3$} surfaces,} \newblock{Ann. Sci. \'Ecole Norm. Sup., 4}
(1974), 543--567.
\bibitem{Ba} H.\ F.\ Baker, \newblock{Principles of geometry,} II,
Cambridge University Press 1922. 
\bibitem{BM1}E.\ Bombieri and D.\ Mumford, \newblock{Enriques' classification of surfaces in char. $p$, II}, 
\newblock{in "Comples Analysis and Algebraic Geoemtry" 
(W.\ L. Bailly and T.\ Shioda, eds.), Iwanami Shoten, Tokyo, and Princeton Univ. Press,
Princeton, NJ, 1977, 22--42.}
\bibitem{BM2}E.\ Bombieri and D.\ Mumford, \newblock{Enriques' classification of surfaces in char. $p$, III}, 
\newblock{Inventiones Math., 35 (1976), 197--232.}
\bibitem{CD}F.\ Cossec and I.\ Dolgachev, \newblock{Enriques surfaces I,} \newblock{Progress in Math., vol. 76, 1989, Birkh\"auser.}
\bibitem{Cr}R.\ M.\ Crew, \newblock{Etale $p$-covers in characteristic $p$,}
\newblock{Compositio Math., 52 (1984), 31--45.}
\bibitem{D}I.\ Dolgachev, \newblock{Numerically trivial automorphisms of Enriques surfaces in arbitrary characteristic,}
\newblock{in "Arithmetic and Geometry of $K3$ surfaces and Calabi-Yau threefolds", Fields Institute Communications 67, 267--283, Springer 2013.}

\bibitem{DK}I.\ Dolgachev and S.\ Kond\=o, \newblock{A supersingular $K3$ surface in characteristic 2 and Leech lattice,}
\newblock{IMRN 2003 (2003), 1--23.}

\bibitem{EHS}T.\ Ekedahl, J.\ M.\ E.\ Hyland and N.\ I.\ Shepherd-Barron, \newblock{Moduli and periods of simply connected Enriques surfaces,} \newblock{arXiv:1210.0342.}

\bibitem{vG}  G.\ van der Geer, \newblock{On the geometry of a Siegel modular threefold}, \newblock{Math. Ann., 260 (1982), 317--350.}

\bibitem{I}J.\ Igusa, \newblock{Betti and Picard numbers of abstract algebraic surfaces,}
\newblock{Proc. Nat. Acad. Sci. U.S.A., 46 (1960), 724--726}
\bibitem{K}T.\ Katsura, \newblock{Surfaces unirationnelles en caract\'eristique $p$,}
\newblock{C. R. Acad. Sc. Paris, t. 288 (1979), 45--47.}
\bibitem{KT}T.\ Katsura and Y.\ Takeda, \newblock{Quotients of abelian and
hyperelliptic surfaces by rational vector fields,}
\newblock{J. Algebra, 124 (1989), 472--492.}
\bibitem{KK}T.\ Katsura and S.\ Kond\=o, \newblock{A note on a supersingular 
$K3$ surface in chrateristic 2,}
\newblock{EMS Series of Congress Reports, in \lq\lq Geometry and Arithmetic"
(C.\ Faber, G.\ Farkas, R.\ de Jong, eds.), 2012, 243--255.}
\bibitem{KU}T.\ Katsura and K.\ Ueno, \newblock{On elliptic surfaces 
in characteristic $p$,}
\newblock{Math. Ann., 272 (1985), 291--330.}
\bibitem{Ko} S.\ Kond\=o, \newblock{Enriques surfaces with finite automorphism groups}, \newblock{Japanese J. Math., 12 (1986), 
191--282.}
\bibitem{L} C.\ Liedtke, \newblock{Arithmetic moduli and liftings of Enriques surfaces,}\newblock{arXiv:1007.0787v2.}

\bibitem{N} V.\ Nikulin, \newblock{On a description of the automorphism groups
of Enriques surfaces,} \newblock{Soviet Math. Dokl., 30 (1984), 282--285.}
\bibitem{RS}A.\ N.\ Rudakov and I.\ R.\ Shafarevich, \newblock{Inseparable morphisms
of algebraic surfaces,}
\newblock{Izv. Akad. Nauk SSSR Ser. Mat., 40 (1976), 1269--1307.}
\bibitem{RS2}A.\ N.\ Rudakov and I.\ R.\ Shafarevich, \newblock{Supersingular surfaces
of type $K3$ over fields of characteristic 2,}
\newblock{Izv. Akad. Nauk SSSR Ser. Mat., 42 (1978), 848--869.}
\bibitem{S}T.\ Shioda, \newblock{An example of unirational surfaces in characteristic $p$,}
\newblock{Math. Ann., 211 (1974), 233--236.}
\bibitem{T}J.\ Tate, \newblock{Algorithm for determining the type of a singular fiber
in an elliptic pencil,}
\newblock{in \lq\lq Modular Functions of One Variable IV", Lecture Notes in Math. 
476, Springer-Verlag, Berlin$\cdot$Heidelberg$\cdot$New York, 1970, 33--52.}
\bibitem{V} E.\ B.\ Vinberg, \newblock{Some arithmetic discrete groups in Lobachevskii spaces}, in \newblock{"Discrete subgroups of Lie groups and applications to Moduli", Tata-Oxford (1975), 323--348.}
\bibitem{V2} E.\ B.\ Vinberg, \newblock{The two most algebraic $K3$ surfaces}, \newblock{Math. Ann., 265 (1983), 1--21.}
\end{thebibliography}
\end{document}